\documentclass[11pt, a4paper]{article}

\usepackage[a4paper, top=2.5cm, bottom=2.5cm, left=2.5cm, right=2.5cm]{geometry}

\usepackage[T1]{fontenc}      % Handles font encoding for accented characters
\usepackage[utf8]{inputenc}     % Allows you to type characters directly from your keyboard
\usepackage[english]{babel}     % Language support
\usepackage{lmodern}          % A modern, high-quality version of the default Computer Modern font

\usepackage{amsmath} % For advanced math environments (align, etc.)
\usepackage{amssymb} % For extended symbol collection
\usepackage{amsthm}  % For theorem-like environments
\usepackage{xparse, tikz, todonotes}  % For creating more powerful custom commands

\usepackage[pagewise]{lineno}
\newcommand{\R}{\mathbb{R}} % Real numbers
\newcommand{\N}{\mathbb{N}} % Natural numbers
\newcommand{\Z}{\mathbb{Z}} % Integers
\newcommand{\diam}{\text{diam}}

\NewDocumentCommand{\dv}{ O{} m m }{ \frac{d\IfNoValueTF{#1}{}{^{#1}}#2}{d#3\IfNoValueTF{#1}{}{^{#1}}} }
\NewDocumentCommand{\pdv}{ O{} m m }{ \frac{\partial\IfNoValueTF{#1}{}{^{#1}}#2}{\partial#3\IfNoValueTF{#1}{}{^{#1}}} }

\newcommand{\abs}[1]{\left\lvert #1 \right\rvert}

\newcommand{\paren}[1]{\left( #1 \right)}

\newcommand*\ordgrow{\overline{\mathbb{O}}}
\newcommand*\ordzero{\textbf{0}}

\newcommand{\floor}[1]{\left\lfloor #1 \right\rfloor}

\newtheorem{theorem}{Theorem}

\newtheorem{corollary}{Corollary}

\newtheorem{lemma}{Lemma}[section]
\newtheorem{proposition}{Proposition}[section]
\newtheorem{claim}{Claim}[section]

\theoremstyle{definition}
\newtheorem{definition}{Definition}[section]

\theoremstyle{remark}
\newtheorem{remark}{Remark}[section]

\newtheorem*{question*}{Question}

\usepackage{booktabs} % For professional-looking tables
\usepackage{hyperref} % For clickable links (e.g., in table of contents, citations)
\hypersetup{
    colorlinks=true,
    linkcolor=red,
    filecolor=magenta,      
    urlcolor=cyan,
    citecolor=blue,
    pdftitle={Math Article},
    pdfpagemode=FullScreen,
}

\usepackage[backend=biber,style=alphabetic]{biblatex}
\title{Complexity of measure-induced maps}
\author{Gabriel Lacerda\thanks{This work was supported by a CNPq grant and by the Fulbright U.S. Student Program, sponsored by the U.S. Department of State and Fulbright Brazil. The author also thanks Enrique R. Pujals and Javier Correa for fruitful discussions. Email: lacerda@im.ufrj.br} \\ Universidade Federal do Rio de Janeiro}
\date{}

\begin{document}

\maketitle

% --- Abstract ---
\begin{abstract}
A classical result by Glasner and Weiss states that a map $f$ has zero topological entropy if and only if its measure-induced map $f_*$, defined as the push-forward of a measure, does. In this work, we use generalized entropy to distinguish the complexity of these maps and prove that the measure-induced map is strictly more complex whenever the generalized entropy of the original map is nonzero and is not an abstract order of growth. Moreover, we show that the complexity of the measure-induced map is of at most polynomial type when the nonwandering set of $f$ is finite. For homeomorphisms of infinite spaces, this bound is attained, so measure-induced maps of Morse--Smale diffeomorphisms cannot be distinguished from one another.

\vspace{0.5cm} \noindent \textbf{Mathematics Subject Classification (2020):} 37B40, 37D15, 28A33.
\end{abstract}

\section{Introduction}

In the field of dynamical systems, topological entropy is a classical tool used to evaluate the complexity of a continuous map $f: M \to M$ on a compact metric space $M$. Denoted by $h(f)$, topological entropy measures the asymptotic exponential growth rate of the number of distinct orbits at small scales. However, for dynamical systems with sub-exponential growth of these distinct orbits, the topological entropy is zero. 

Positive topological entropy is often used to characterize systems that are chaotic or complex. However, there can be distinct maps with zero entropy where one is inherently more chaotic than the other. To illustrate, within the class of Morse--Smale dynamics---a classical example of zero-entropy maps---two distinct maps can have different phase diagrams \cite{Palis_1969}. This property implies that, depending on the phase space $M$, the behavior of their orbits can be completely different, and one can be strictly more complex than the other.

To address this situation, Correa and Pujals introduced the concept of generalized entropy \cite{Correa_Pujals_2023}, an approach that considers the growth rate of distinct orbits at all scales. By ``growth rate,'' we refer to the order of growth of the corresponding sequence, rather than a numerical asymptotic rate. Unlike topological entropy, generalized entropy is not a non-negative number; it is an equivalence class in a partially ordered set. More precisely, considering non-decreasing functions $a : \N \to [1, \infty)$, we define $\mathbb{O}$ as the set of equivalence classes of such functions, where two functions are equivalent if they have the same growth rate. Hence, an element in $\mathbb{O}$ is denoted by $[a(n)]$. The class corresponding to bounded growth is denoted by $\ordzero$, which represents the smallest possible growth. Furthermore, one can complete this partially ordered set to obtain the complete set of orders of growth $\ordgrow$. This extended set contains abstract orders of growth, which are given by the suprema or infima of subsets of $\mathbb{O}$. Therefore, the generalized entropy of a map $f: M \to M$ is denoted by $o(f)$ and is an element of $\ordgrow$. In this sense, generalized entropy provides a broader notion of \emph{complexity}.

This work is concerned with the generalized entropy of the measure-induced map $f_*: \mathcal{P}(M) \to \mathcal{P}(M)$, where $\mathcal{P}(M)$ is the set of Borel probability measures on $M$ and $f_*$ is the push-forward operator. That is, for $\mu \in \mathcal{P}(M)$, $f_*\mu = \mu\circ f^{-1}$. A classical result regarding the complexity of the measure-induced map was established by Glasner and Weiss in \cite{Glasner_Weiss_1995}, where they showed that the topological entropy of $f$ is zero if and only if the topological entropy of $f_*$ is zero. Since generalized entropy serves as a tool to distinguish dynamical systems with zero topological entropy, a natural question arises: can it also capture the difference in complexity between the base map $f$ and its measure-induced map $f_*$?

We answer this question in the affirmative by showing that the measure-induced map is significantly more complex. In fact, we provide a generalization of a well-known result by Bauer and Sigmund \cite{Bauer_Sigmund_1975}, which states that if the topological entropy of $f$ is positive, then the topological entropy of $f_*$ is infinite. The strict growth of complexity, $o(f_*) > o(f)$, does not occur for maps with the lowest ($o(f) = \ordzero$) and highest complexity ($o(f) = \sup \mathbb{E}$, see \S\ref{subsec:orders-of-growth}). More precisely, whenever $o(f) \neq \ordzero$, $o(f_*)$ is an abstract order of growth; that is, it lies strictly in the completion $\ordgrow \setminus \mathbb{O}$. Moreover, given a one-parameter family of orders of growth (e.g., $\mathbb{A} = \{[a(n)^t] \in \mathbb{O}: t > 0\}$ for a given $[a(n)] \in \mathbb{O}$), we show that the complexity of $f_*$ is at least the supremum of this family whenever $o(f_*)$ is bounded below by a base element of $\mathbb{A}$. Our complete first result is the following:

\begin{theorem}\label{thm:supremum-power-set-measure}
    Let $f: M \to M$ be a continuous map on a compact metric space.
    \begin{enumerate}
        \item $o(f) = \emph{\ordzero}$ if, and only if, $o(f_*) = \emph{\ordzero}$. Otherwise, $o(f_*) \in \ordgrow\setminus\mathbb{O}$.
        \item For $\mathbb{A} = \{[a(n)^t] \in \mathbb{O} : t > 0\}$, if $o(f_*) \geq [a(n)]$, then $o(f_*) \geq \sup \mathbb{A}$.
        \item For $\mathbb{B} = \{[b(n)^t] \in \mathbb{O} : t > 0\}$, if $[b(n)] \geq o(f_*)$, then $\inf \mathbb{B} \geq o(f_*)$.
        \item In particular, if $o(f) \geq [a(n)]$, then $o(f_*) \geq \sup \mathbb{A}$. Furthermore, if $o(f) \in \mathbb{O}$ is such that $o(f) > \emph{\ordzero}$, then $o(f_*) > o(f)$.
    \end{enumerate}
\end{theorem}

The first statement of the preceding result is a generalized-entropy analogue of the main theorem in \cite{Glasner_Weiss_1995}, because $\ordzero$ is the least element of the partially ordered set $\ordgrow$. A consequence of the above theorem relates to the polynomial one-parameter family of orders of growth $\mathbb{P} = \{[n^t] \in \mathbb{O}: t > 0\}$, illustrated in Figure \ref{fig:polynomial-one-parameter-family}. Indeed, we can apply it in particular to homeomorphisms of the circle, where \cite{Correa_Pujals_2023} showed that every circle homeomorphism that is not Lyapunov-stable has linear complexity $[n]$ (e.g., Morse--Smale diffeomorphisms). Since Lyapunov-stable dynamics are tame, for homeomorphisms of the circle there are only two possibilities: either the complexity of the measure-induced map is $\ordzero$, or it is at least $\sup \mathbb{P}$, which we call \emph{super-polynomial}. Moreover, we can use \cite[Theorem 3]{Correa_Pujals_2023} to deduce the following result. Recall that the $\alpha$-limit set of a point $x \in M$ is the set $\alpha(x)$ of accumulation points of the backward orbit of $x$ under $f$.

\begin{corollary}\label{coro:alpha-limit-sup-poly}
    If a homeomorphism $f: M \to M$ has a point $x \in M$ such that $x \notin \alpha(x)$, then $o(f_*) \geq \sup \mathbb{P}$. In particular, if $M$ is the circle $S^1$, then either $o(f_*) \geq \sup \mathbb{P}$ or $o(f_*) = \emph{\ordzero}$.
\end{corollary}

\begin{figure}[ht]
        \centering
        \begin{tikzpicture}[scale=1]
        \draw[line width=2pt,blue] (0,0) -- (5,0);
        
        \fill[red] (0,0) circle (3pt);
        \node[below=7pt] at (0,0) {$\inf\mathbb{P}$};

        \fill[red] (1,0) circle (3pt);
        \node[below=7pt] at (1,0) {$[n^t]$};

        \fill[red] (2,0) circle (3pt);
        \node[below=7pt] at (2,0) {$[n]$};

        \fill[red] (3,0) circle (3pt);
        \node[below=7pt] at (3,0) {$[n^2]$};

        \fill[red] (5,0) circle (3pt);
        \node[below=7pt] at (5,0) {$\sup\mathbb{P}$};
        \end{tikzpicture}
        \caption{The one-parameter family of polynomial orders of growth can be represented as a line in the space of orders of growth. Moreover, by taking the completion of this set, the infimum and supremum can be identified as the endpoints of the family.}
        \label{fig:polynomial-one-parameter-family}
\end{figure}

The regime of positive topological entropy is related to the exponential one-parameter family of orders of growth $\mathbb{E} = \{[e^{tn}] \in \mathbb{O} : t > 0\}$. Specifically, there is a natural way to project the generalized entropy of a map $f$ onto this family by considering the set $I_\mathbb{E}(f) = \{ t > 0 : o(f) \leq [e^{tn}]\} \subset \mathbb{R}$ and defining the projection map $\pi_\mathbb{E}: \ordgrow \to [0, \infty]$ by $\pi_\mathbb{E}(o(f)) = \inf I_\mathbb{E}(f)$ if this set is non-empty, and $\pi_\mathbb{E}(o(f)) = \infty$ otherwise. Justifying this projection, by Theorem 2 in \cite{Correa_Pujals_2023}, we have $\pi_\mathbb{E}(o(f)) = h(f)$.

Consequently, one can extend the notion of topological entropy to the topological entropy with respect to the one-parameter family $\mathbb{A}$, given by the projection $\pi_\mathbb{A}$, and this quantity is denoted by $h_\mathbb{A}(f)$.

\begin{corollary}\label{coro:zero-or-infinite-entropy}
    If $h_\mathbb{A}(f_*) > 0$, then $h_\mathbb{A}(f_*) = \infty$. In other words, $h_\mathbb{A}(f_*) \in \{0, \infty\}$.
\end{corollary}

When the one-parameter family is the polynomial family $\mathbb{P}$, $h_\mathbb{P}(f)$ coincides with the polynomial entropy, a tool useful for distinguishing dynamical systems with zero topological entropy, which was revisited by Marco in \cite{Marco_2013}. In other words, whenever the polynomial entropy of $f$ is positive, the polynomial entropy of $f_*$ is infinite. Therefore, generalized entropy is a much broader tool for differentiating maps with vanishing topological entropy.

Conversely, it is natural to ask whether one can establish upper bounds for the generalized entropy of the measure-induced map. Correa and de Paula established in \cite{Correa_dePaula_2023} that the generalized entropy of a homeomorphism $f$ with a finite nonwandering set always satisfies $o(f) \leq \sup \mathbb{P}$. The following result recovers this bound, since $o(f_*) \geq o(f)$, without relying on their technique of coding wandering points. Instead, we use a different kind of coding that discretizes the mass of the iterates of a measure $\mu$ under $f_*$ with respect to a partition of $\mathcal{P}(M)$, and we show that the number of such discretizations is at most polynomial. Recall that the nonwandering set $\Omega(f)$ is the set of points $x \in M$ for which every neighborhood of $x$ intersects its image under some positive iterate of $f$. Ultimately, we show that this polynomial bound can be extended to the measure-induced map $f_*$.

\begin{theorem}\label{thm:omega-finite-measure-induced-sup-poly}
    Let $f: M \to M$ be a homeomorphism such that $\Omega(f)$ is finite. Then $o(f_*) \leq \sup \mathbb{P}$.
\end{theorem}

We can combine the above results to provide sufficient conditions for the measure-induced map of a homeomorphism to attain exactly the supremum of the polynomial one-parameter family. As a consequence, generalized entropy is more effective than polynomial entropy for distinguishing maps with vanishing topological entropy, because the following result is more precise than simply saying that a map has infinite polynomial entropy.

\begin{corollary}
    If $f$ is a homeomorphism on an infinite compact metric space $M$ such that $\Omega(f)$ is finite, then $o(f_*) = \sup \mathbb{P}$.
\end{corollary}

This holds because a homeomorphism with a finite nonwandering set on an infinite compact metric space has a wandering point $x$, and a wandering point satisfies $x\notin\alpha(x)$. In particular, Corollary~3 applies to Morse--Smale diffeomorphisms of compact manifolds of positive dimension. For such diffeomorphisms on compact surfaces, Correa and de Paula showed that $o(f)=[n^{L}]$, where $L=L(f)$ is an integer determined by the graph associated with $f$ in \cite[Theorem 1]{Correa_dePaula_2023}, and this integer can vary from one diffeomorphism to another. Nevertheless, all of their measure-induced maps satisfy $o(f_*)=\sup\mathbb{P}$, so $f_*$ cannot distinguish them.

Returning to the context of homeomorphisms of the circle, we do not know the complexity of the measure-induced map of the Denjoy homeomorphism, because its nonwandering set is infinite. We know that the complexity of the base map is linear, but can the complexity of its measure-induced map be strictly greater than $\sup \mathbb{P}$? Or, more generally:

\begin{question*}
    Does there exist a continuous map $f$ such that $o(f) \in \mathbb{P}$, but $o(f_*) > \sup \mathbb{P}$?
\end{question*}

\textbf{This work is organized as follows.} Section \ref{section:fundamentals} recalls the dynamical and measure-theoretic background, including the space of probability measures, measure-induced maps, orders of growth, and generalized entropy. In \S\ref{section:generalized-entropy-measure-induced}, we study the generalized entropy of measure-induced maps. Specifically, in \S\ref{subsec:proof-thm-A} we provide the proofs of Theorem \ref{thm:supremum-power-set-measure} and Corollary \ref{coro:zero-or-infinite-entropy}, and in \S\ref{subsec:proof-thm-B} we prove Theorem \ref{thm:omega-finite-measure-induced-sup-poly}.

\section{Fundamentals}\label{section:fundamentals}

We first recall the dynamical notions that will be used throughout the paper. We then introduce the probability space and the measure-induced map, and finally recall the generalized entropy and the orders of growth that will be used to measure the complexity of the induced dynamics.

From now on, we denote a compact metric space by $(M, d)$, or simply $M$ when the metric $d$ is clear from the context. The \emph{diameter} of $M$, denoted by $\operatorname{diam}(M)$, is the supremum of the distances between pairs of points in $M$ with respect to the metric $d$, which is finite because $M$ is compact. As discussed earlier, $f: M \to M$ is a continuous map. To avoid excessive use of parentheses, we may also denote the image of $x \in M$ under $f$ by $fx$. A closed subset $X \subseteq M$ is said to be $f$-\emph{invariant} if $f(x) \in X$ for every $x \in X$; that is, $f(X) \subseteq X$.

An important set in the study of dynamical systems is the \emph{nonwandering set} of $f$, denoted by $\Omega(f)$. A point $x$ belongs to $\Omega(f)$ if, for every neighborhood $U$ of $x$ in $M$, there exists an integer $n \geq 1$ such that $f^n(U) \cap U \neq \varnothing$. Since $M$ is compact, the nonwandering set is nonempty, compact, and invariant. In a sense, nonwandering points are those that possess a recurrent open neighborhood. A point that is not in the nonwandering set is called a \emph{wandering point}. Note that for each wandering point $x$, there exists a \emph{wandering neighborhood} $U$ of $x$; that is, an open set $U$ containing $x$ such that $f^n(U) \cap U = \varnothing$ for all $n \geq 1$.

Another important set is the set of periodic points $\operatorname{Per}(f) = \{x \in M : \exists m \in \N, f^m(x) = x\}$. Note that the set of fixed points $\operatorname{Fix}(f) = \{x \in M : f(x) = x\}$ is contained in $\operatorname{Per}(f)$, which in turn is contained in the nonwandering set $\Omega(f)$. Finally, if $f$ is invertible, we define the \emph{$\alpha$-limit set} of $x \in M$ as
\begin{equation*}
    \alpha(x) = \left\{y \in M : \text{there exists a sequence } n_k \to \infty \text{ such that } f^{-n_k}(x) \to y\right\}.
\end{equation*}

Recall that a map $f$ is \emph{Lyapunov stable} if for every $\varepsilon > 0$, there exists $\delta > 0$ such that for any pair of points $x, y \in M$ with $d(x, y) < \delta$, we have $d(f^nx, f^ny) < \varepsilon$ for all $n \geq 0$.

For a given $n \in \N$, the \emph{dynamical distance} on $M$ is defined as 
\begin{equation}\label{eq:dynamical-distance}
    d_n(x, y) = \max\{d(f^ix, f^iy) : 0 \leq i < n\}.
\end{equation}
It is not hard to see that $(M, d_n)$ remains a compact metric space and that $d_n$ defines the same topology as $d$ on $M$. For $\varepsilon > 0$, we say that $A \subset M$ is an \emph{$(n, \varepsilon)$-separated set} if $d_n(x, y) \geq \varepsilon$ for any pair of distinct points $x, y \in A$. Note that the cardinality of $A$ is finite because $M$ is compact under the metric $d_n$. We denote by $\operatorname{Sep}(f, n, \varepsilon)$ the maximal cardinality of an $(n, \varepsilon)$-separated subset of $M$. This quantity measures the orbit complexity of $f$ at scale $\varepsilon$ over $n$ iterates. Observe that, for any given $\varepsilon > 0$, the sequence $\operatorname{Sep}(f, n, \varepsilon)$ is non-decreasing with respect to $n$. The \emph{topological entropy} of $f$ is defined as
\begin{equation*}
    h(f) = \lim\limits_{\varepsilon \to 0}\limsup_{n \to \infty} \frac{\log \operatorname{Sep}(f, n, \varepsilon)}{n}.
\end{equation*}
A complete exposition of the topological entropy is given in \cite[Chapter 3]{Katok_Hasselblatt_1995}. The \emph{polynomial entropy} was introduced by J.P. Marco in \cite{Marco_2013} and later studied by C. Labrousse and J.P. Marco in \cite{Labrousse_Marco_2014}, and is given by:
\begin{equation}
    h_{\text{pol}}(f) = \lim\limits_{\varepsilon \to 0}\limsup\limits_{n \to \infty} \frac{\log \text{Sep}(f, n, \varepsilon)}{\log n}.
\end{equation}

For any set $E$, let $\#E$ denote its cardinality. If $\varphi: M \to \R$ is a continuous function and $E$ is a subset of $M$, set $\operatorname{osc}_E \varphi = \sup_{x \in E} \varphi(x) - \inf_{x \in E} \varphi(x)$. Moreover, let $\mathbf{1}_E$ denote the indicator (or characteristic) function of $E$, which need not be continuous.

\subsection{Probability space and the measure-induced map}\label{subsec:probability-space-measure-induced}

Let $\mathcal{P}(M)$ denote the space of Borel probability measures on $M$, endowed with the weak-$*$ topology. This compact metrizable space is frequently called the \emph{probability space of $M$} (for a comprehensive overview, see \cite[Chapter 2]{Denker_Grillenberger_Sigmund_1976}). The \emph{measure-induced map} $f_*: \mathcal{P}(M) \to \mathcal{P}(M)$ is defined by the push-forward action $\mu \mapsto \mu\circ f^{-1}$, as illustrated in Figure \ref{fig:measure-induced-map}. As shown in \cite[Proposition 3.2]{Denker_Grillenberger_Sigmund_1976}, $f_*$ is a continuous affine map. The push-forward measure $f_*\mu$ is also characterized by the relation $\int\phi\,d(f_*\mu)=\int\phi\circ f\,d\mu$ for any bounded measurable function $\phi: M \to \R$.

%O espaço de probabilidade não é um hiperespaço, logo continua com mathcal

\begin{figure}[ht]
        \centering
        \begin{tikzpicture}[scale=1]
  % ======================
  % BLOCO ESQUERDO: X
  % ======================

  % --- Segmento superior X
  \draw[line width=2pt,blue] (0,0) -- (5,0);
  \node[below right=3pt] at (4.7,0) {\footnotesize $M$};

  % --- Curva e preenchimento
  \begin{scope}
    \path[fill=red!20,opacity=0.55]
      (0,0)
      -- plot[domain=0:5,samples=200,smooth]
         (\x,{1 + 0.4*sin(320*\x) * exp(-(\x-2.5)^2/7) + 0.2*exp(-(\x-1)^2/0.8)})
      -- (5,0) -- cycle;
  \end{scope}
  \draw[red!85!black,thick,smooth,domain=0:5,samples=200]
    plot (\x,{1 + 0.4*sin(320*\x) * exp(-(\x-2.5)^2/7) + 0.2*exp(-(\x-1)^2/0.8)});

    \draw[line width=2pt, purple] (3,0) -- (3.5, 0);
  \draw[line width=2pt, purple] (3, -0.1) -- (3, 0.1);
  \draw[line width=2pt, purple] (3.5, -0.1) -- (3.5, 0.1);

  % --- Linhas tracejadas nas pontas
  \draw[dashed,gray] (0,0) -- (0,1.05);
  \draw[dashed,gray] (5,0) -- (5,1.05);

  % --- Segmento inferior (cópia)
  \draw[line width=2pt,blue] (0,-2) -- (5,-2);
  \node[below right=3pt] at (4.7,-2) {\footnotesize $M$};
  \draw[line width=2pt, purple] (3,-2) -- (4, -2);
  \draw[line width=2pt, purple] (3, -2.1) -- (3, -1.9);
  \draw[line width=2pt, purple] (4, -2.1) -- (4, -1.9);
  \node[below=1pt] at (3.5,-2) {\footnotesize $B$};

  % --- Seta preta curvada para baixo
  \draw[->,black,thick]
    (0.5,-0.15) .. controls +(-0.3,-0.6) and +(-0.3,0.6) .. (0.5,-1.85);
    \node at (0.6,-1) {$f$};

  % ======================
  % BLOCO DIREITO: T(X)
  % ======================

  \begin{scope}[shift={(7,0)}]
    % --- Segmento superior
    \draw[line width=2pt,blue] (0,0) -- (5,0);
    \node[below right=3pt] at (4.7,0) {\footnotesize $M$};

    % --- Curva e preenchimento (diferente)
    \begin{scope}
      \path[fill=cyan!20,opacity=0.55]
        (0,0)
        -- plot[domain=0:5,samples=200,smooth]
           (\x,{0.9 + 0.5*sin(190*\x) * exp(-(\x-2.5)^2/5) + 0.3*exp(-(\x-3.5)^2/1)})
        -- (5,0) -- cycle;
    \end{scope}
    \draw[cyan!85!black, thick,smooth,domain=0:5,samples=200]
      plot (\x,{0.9 + 0.5*sin(190*\x) * exp(-(\x-2.5)^2/5) + 0.3*exp(-(\x-3.5)^2/1)});

    % --- Linhas tracejadas nas pontas
    \draw[dashed,gray] (0,0) -- (0,0.9);
    \draw[dashed,gray] (5,0) -- (5,0.8);

  \end{scope}

  % ======================
  % SETA ENTRE BLOCOS
  % ======================

  \draw[->,black,thick]
    (5.3,0.4) .. controls +(0.3,-0.4) and +(-0.3,0) .. (6.7,0.2);

  % --- rótulo opcional da seta (remova se quiser mais limpo)
  \node at (6,0.45) {$f_*$};
        \end{tikzpicture}
        \caption{The red measure, represented as a density, is mapped by $f_*$ to the blue measure (also represented as a density) by taking the preimages of elements of the Borel $\sigma$-algebra of $M$ under $f$.}
        \label{fig:measure-induced-map}
\end{figure}

Consider the Banach space of continuous real-valued functions on $M$, denoted by $C^0(M)$, endowed with the supremum norm:
\begin{equation*}
    \lVert \varphi \rVert_\infty = \max_{x \in M} \abs{\varphi(x)}.
\end{equation*}
Let $(\varphi_k)_{k \in \mathbb{N}}$ be a dense sequence in the unit ball of $C^0(M)$. We define the following metric on $\mathcal{P}(M)$:
\begin{equation*}
    \rho(\mu, \nu) = \sum\limits_{k \in \N} \frac{\lvert \int \varphi_k d\mu - \int \varphi_k d\nu\rvert}{2^k},
\end{equation*}
where $\mu, \nu \in \mathcal{P}(M)$. Moreover, observe that the diameter of $\mathcal{P}(M)$ with respect to the metric $\rho$ is at most $2$. It is well known that this metric induces the weak-$*$ topology on $\mathcal{P}(M)$ \cite{Sigmund_1978}.

Finally, let $\mathcal{P}_1(M)$ denote the space of Dirac measures $\{\delta_x : x \in M\}$. Note that this subspace is both compact and invariant under $f_*$. Furthermore, the evaluation map $i: \mathcal{P}_1(M) \to M$ defined by $i(\delta_x) = x$ is a homeomorphism, and we naturally have $f_*(\delta_x) = \delta_{f(x)}$ for all $x \in M$.

\begin{remark}\label{rmk:iterations-measure-induced-map}
    A useful fact that we will need later is that $(f^n)_* = (f_*)^n$. Indeed, for any measure $\mu \in \mathcal{P}(M)$ and any element $B$ of the Borel $\sigma$-algebra of $M$, $(f^n)_*\mu(B) = \mu(f^{-n}B) = (f_*)^n\mu(B)$.
\end{remark}

\subsection{The complete set of orders of growth and generalized entropy}\label{subsec:orders-of-growth}

Consider the set of non-decreasing sequences of real numbers assuming values in $[1, \infty)$
\begin{equation*}
    \mathcal{O} = \{a(n) : a(n) \leq a(n + 1), \forall n \in \mathbb{N}\}.
\end{equation*}
On this set, we define the following equivalence relation: Given $a_1, a_2 \in \mathcal{O}$, we say that $a_1 \approx a_2$ if there exist $c, C \in (0, +\infty)$ such that $ca_1(n) \leq a_2(n) \leq Ca_1(n)$ for all $n \in \mathbb{N}$. This means that two sequences are related if they have the same order of growth. Thus, the quotient space $\mathbb{O} = \mathcal{O}/{\approx}$ is called \emph{the set of orders of growth}. A class belonging to $\mathbb{O}$ will be denoted by $[a(n)]$. We can define a partial order on $\mathbb{O}$ based on the idea that one order of growth can be faster than another. In fact, given $[a(n)], [b(n)] \in \mathbb{O}$, we say that $[a(n)] \leq [b(n)]$ if there exists $C > 0$ such that $a(n) \leq Cb(n)$ for all $n \in \mathbb{N}$. This partial order is well-defined because it does not depend on the choice of representatives $a(n)$ and $b(n)$. Hence, $\mathbb{O}$ is also called a poset. Note that all bounded non-decreasing sequences belong to the constant class $[L]$ for some fixed natural number $L$. This is a special class of orders of growth, as all other classes are greater than this one, and it will be denoted by $\ordzero$. Let us present some equivalent criteria for determining whether one class grows faster than another.

\begin{proposition}[\cite{Correa_Pujals_2023}]\label{lemma:order_growth_equivalent_criteria}
    The following four are equivalent:
    \begin{enumerate}
        \item[\emph{I)}] $[a_1(n)] \leq [a_2(n)]$, or there exists $c > 0$ such that $a_1(n) \leq c\cdot a_2(n)$ for all $n \in \mathbb{N}$;
        \item[\emph{II)}] $\liminf_n \frac{a_2(n)}{a_1(n)} > 0$;
        \item[\emph{III)}] $\limsup_n \frac{a_1(n)}{a_2(n)} < \infty$;
        \item[\emph{IV)}] There exists $c > 0$ and $n_0 \in \mathbb{N}$ such that $a_1(n) \leq c\cdot a_2(n)$ for all $n \geq n_0$.
    \end{enumerate}
\end{proposition}

Two important subsets of $\mathbb{O}$ will be presented: the families of polynomial and exponential orders of growth. The first is given by $\mathbb{P} = \{[n^t] \in \mathbb{O} : t \in (0, \infty)\}$, illustrated in Figure \ref{fig:polynomial-one-parameter-family}, and represents all possible polynomial orders of growth. The latter is defined as $\mathbb{E} = \{[\exp{(tn)}] \in \mathbb{O} : t \in (0, \infty)\}$. Understanding the structure of the sets $\mathbb{P}$ and $\mathbb{E}$ within $\mathbb{O}$ naturally motivates the investigation of their supremum and infimum. Recall that a set $S$ endowed with a partial order is a complete lattice if every subset $A \subset S$ has both an infimum and a supremum.

\begin{definition}[\cite{Correa_Pujals_2023}]
    The smallest complete lattice that contains $\mathbb{O}$, or the Dedekind--MacNeille completion of $\mathbb{O}$, denoted by $\ordgrow$, is called \emph{the complete set of orders of growth}.
\end{definition}

In particular, it is uniquely defined, and from now on, we will consider $\mathbb{O} \subset \ordgrow$. Note that $\sup \mathbb{P}$ and $\sup \mathbb{E}$ belong only to the complete set of orders of growth; that is, they are abstract orders of growth that do not belong to the space of orders of growth. We also use $o$ to denote a general abstract order of growth in $\ordgrow$.

%\begin{remark}\label{remark:multiplication-order-growth}
%    The pointwise multiplication of orders of growth is again an order of growth. Indeed, if $[a(n)]$ and $[b(n)]$ belong to $\mathbb{O}$, then their product $[a(n)]\cdot [b(n)]$, defined as $[a(n) \cdot b(n)]$, also belongs to $\mathbb{O}$. This holds because both representatives $a(n)$ and $b(n)$ of these classes are non-decreasing sequences such that $a(n) \geq 1$ and $b(n) \geq 1$. Therefore, $a(n)\cdot b(n)$ is again a non-decreasing sequence bounded below by $1$.
%\end{remark}

%\begin{definition}
%    Given two abstract orders of growth $o_1, o_2 \in \ordgrow\setminus\mathbb{O}$ such that $o_1 = \sup A$ and $o_2 = \sup B$, where both $A$ and $B$ are subsets of $\mathbb{O}$, we define their product by
%    \begin{equation*}
%        o_1\cdot o_2 = \sup (A \cdot B) = \sup \{[a(n)] \cdot [b(n)] : [a(n)] \in A \text{ and } [b(n)] \in B\}.
%    \end{equation*}
%\end{definition}

%In particular, if $o \in \ordgrow$ is given by $o = \sup A$, then $o^2 = \sup (A \cdot A)$. 

Recall that, for any fixed $\varepsilon > 0$, the sequence $\operatorname{Sep}(f, n, \varepsilon)$ is non-decreasing with respect to $n$. Moreover, observe that $[\operatorname{Sep}(f, n, \varepsilon_1)] \geq [\operatorname{Sep}(f, n, \varepsilon_2)]$ if $\varepsilon_2 > \varepsilon_1$. The necessity of the completion $\ordgrow$ arises because the supremum of the growth rates obtained from different scales need not be represented by a single sequence in $\mathbb{O}$. Hence, the \emph{generalized entropy} of $f$ is defined as 
\begin{equation*}
    o(f) = \sup \{[\operatorname{Sep}(f, n, \varepsilon)] \in \mathbb{O} : \varepsilon > 0\} \in \ordgrow.
\end{equation*}
As shown in \cite{Correa_Pujals_2023}, $o(f) \leq \sup \mathbb{E}$ for every map $f: M \to M$. Consequently, Figure \ref{fig:complete-set-generalized-entropy} illustrates the set of all possible classes for the generalized entropy contained within $\ordgrow$. 

\begin{figure}[ht]
        \centering
        \begin{tikzpicture}[scale=1]
        \draw[line width=1pt] (0,0) -- (2,2);
        \draw[line width=1pt] (0,0) -- (2,-2);
        \draw[line width=1pt] (2,2) -- (7,2);
        \draw[line width=1pt] (2,-2) -- (7,-2);
        \draw[line width=1pt] (7,2) -- (9,0);
        \draw[line width=1pt] (7,-2) -- (9,0);

        \fill[red] (0,0) circle (3pt);
        \node[below=4pt] at (0,0) {$\textbf{0}$};

        \draw[dashed,line width=1pt] (5.5,2) -- (5.5,-2);

        \draw[line width=1pt, blue] (5.5,0) -- (9,0);

        \fill[red] (5.5,0) circle (2pt);
        \node[left=4pt] at (5.5,0) {$\inf\mathbb{E}$};

        \fill[red] (9,0) circle (2pt);
        \node[right=4pt] at (9,0) {$\sup\mathbb{E}$};

        \fill[red] (7,0) circle (2pt);
        \node[below=4pt] at (7,0) {$[e^{tn}]$};

        \node[below=3pt, scale = 0.8] at (6.4,2) {$h(f) > 0$};
        \node[below=3pt, scale = 0.8] at (3.5,2) {$h(f) = 0$};

        \fill[red] (1.5,0.5) circle (2pt);
        \node[below=4pt] at (1.5,0.5) {$[\log n]$};

        \fill[red] (3.5,1) circle (2pt);
        \node[below=4pt] at (3.5,1) {$[n\log n]$};

        \fill[red] (2.5,-1) circle (2pt);
        \node[below=4pt] at (2.5,-1) {$[n]$};

        \fill[red] (3.7,-0.5) circle (2pt);
        \node[below=4pt] at (3.7,-0.5) {$[n^2]$};
        \end{tikzpicture}
        \caption{Since the complete set of orders of growth can only be endowed with a partial order, it is depicted as a two-dimensional figure. Moreover, this figure illustrates all possible classes for the generalized entropy of a map $f$. It also shows how much more complexity is captured by the definition of generalized entropy.}
        \label{fig:complete-set-generalized-entropy}
\end{figure}

A crucial property of generalized entropy is that it is a dynamical invariant.

\begin{proposition}[\cite{Correa_Pujals_2023}]\label{prop:generalized_entropy_dynamical_invariant}
Let $M$ and $N$ be compact metric spaces, and let $f: M \to M$ and $g: N \to N$ be continuous maps. Suppose there exists a homeomorphism $h: M \to N$ such that $h \circ f = g \circ h$. Then $o(f) = o(g)$.
\end{proposition}

In other words, when there exists such a homeomorphism $h: M \to N$, we say that $f$ and $g$ are topologically conjugate.

\begin{remark}\label{rmk:iterate-f-generalized-entropy}
    It is not hard to see that, for any $n \in \N$, $o(f^n) \geq o(f)$, as shown in \cite[Proposition B.1]{Correa_Pujals_2023}. In particular, $o((f^n)_*) \geq o(f_*)$.
\end{remark}

For $[a(n)] \in \mathbb{O}$ and $\mathbb{A} = \{[a(n)^t] \in \mathbb{O} : t > 0\}$, we can consider the projection of $o(f)$ into $\mathbb{A}$. This projection is a nonnegative number given by
\begin{equation*}
    h_{\mathbb{A}}(f) = \lim_{\varepsilon \to 0}\limsup_{n \to \infty} \frac{\log \operatorname{Sep}(f, n, \varepsilon)}{\log a(n)}.
\end{equation*}
Another way to characterize this projection is to consider the set $I_\mathbb{A}(f) = \{t > 0 : o(f) \leq [a(n)^t]\} \subset \R$ and define the projection map $\pi_\mathbb{A}: \ordgrow \to [0, \infty]$ by $\pi_\mathbb{A}(o(f)) = \inf I_\mathbb{A}(f)$ if $I_\mathbb{A}(f)$ is nonempty, and $\pi_\mathbb{A}(o(f)) = \infty$ otherwise. The argument of Theorem 2 in \cite{Correa_Pujals_2023} guarantees that $h_\mathbb{A}(f) = \pi_\mathbb{A}(o(f))$. We conclude this section by recalling further results by Correa and Pujals.

\begin{proposition}[\cite{Correa_Pujals_2023}]\label{prop:homeomorphism-linear-growth}
    Let $f: M \to M$ be a homeomorphism. If there exists $x \in M$ such that $x \notin \alpha(x)$, then $o(f) \geq [n]$. Moreover, if $M$ is the circle $S^1$, then either $o(f) = \emph{\ordzero}$ or $o(f) = [n]$.
\end{proposition}

\section{Generalized entropy of measure-induced maps}\label{section:generalized-entropy-measure-induced}

In this section, we present some results concerning the measure-induced map $f_*$ and provide the proofs of Theorems \ref{thm:supremum-power-set-measure} and \ref{thm:omega-finite-measure-induced-sup-poly}.

\begin{proposition}\label{prop:measure-induced-properties}
    If $f: M \to M$ is a continuous map defined on a compact metric space $M$, then the following holds:
    \begin{enumerate}
        \item $o(f_*) \geq o(f)$;
        \item If $\Omega(f_*) = \mathcal{P}(M)$, then $\Omega(f) = M$.
    \end{enumerate}
\end{proposition}
\begin{proof}
    We will show each item.
    \begin{enumerate}
        \item Recall that the map $i: \mathcal{P}_1(M) \to M$ defined by $i(\delta_x) = x$ is a homeomorphism. Furthermore, it conjugates the restriction of $f_*$ to $\mathcal{P}_1(M)$ to the map $f$. Therefore, by Proposition \ref{prop:generalized_entropy_dynamical_invariant}, $o(f_*|_{\mathcal{P}_1(M)}) = o(f)$. Since the separation of orbits in $\mathcal{P}_1(M)$ implies the separation of orbits in $\mathcal{P}(M)$, $o(f_*) \geq o(f_*|_{\mathcal{P}_1(M)})$.

        \item Consider $\mu \in \mathcal{P}(M)$. By a proposition in \cite[p. 418]{Sigmund_1978}, if $\mu$ is a nonwandering point for $f_*$, then $\mu(\Omega(f)) \geq 1/2$. Now, assume the hypothesis that $\Omega(f_*) = \mathcal{P}(M)$. Then, for any $x \in M$, the atomic measure $\delta_x$ is a nonwandering point. Applying the proposition, we have $\delta_x(\Omega(f)) \geq 1/2$. Since $\delta_x$ takes values only in $\{0, 1\}$, this implies $\delta_x(\Omega(f)) = 1$, which means $x \in \Omega(f)$.
    \end{enumerate}
\end{proof}

In other words, the second part of the above proposition says that if the nonwandering set of $f$ is strictly contained in $M$, then the same holds for the measure-induced map $f_*$. This idea will be used heuristically in the proof of Theorem \ref{thm:omega-finite-measure-induced-sup-poly}.

\subsection{Proof of Theorem \ref{thm:supremum-power-set-measure}}\label{subsec:proof-thm-A}

For the next lemma, recall that the distance $\rho$ defined in \S\ref{subsec:probability-space-measure-induced} induces the weak$^*$ topology on $\mathcal{P}(M)$.

\begin{lemma}\label{lemma:sigmund-78-square-separate}
    For every $\varepsilon > 0$, there exists $\delta \in (0, \varepsilon)$ such that for all $\varepsilon_0 \in (0, \delta)$ and for all $n \in \mathbb{N}$, $\operatorname{Sep}(f_*, n, \varepsilon_0) \geq \operatorname{Sep}(f_*, n, \varepsilon)^2$.
\end{lemma}
\begin{proof}
    Choose $b > 0$ such that $b \leq \frac{\varepsilon}{4 + \varepsilon}$, set $\delta := b\varepsilon < \varepsilon$, and fix $\varepsilon_0 \in (0, \delta)$. Given $n \in \mathbb{N}$, let $E \subset \mathcal{P}(M)$ be an $(n, \varepsilon)$-separated set of maximal cardinality; that is, $\#E = \operatorname{Sep}(f_*, n, \varepsilon)$. Define the auxiliary set
    \begin{equation*}
        E_b = \{\lambda \in \mathcal{P}(M) : \lambda = (1 - b)\mu_1 + b\mu_2 \text{ for } \mu_1, \mu_2 \in E\}.
    \end{equation*}
    We will show that for any two distinct ordered pairs $(\mu_1, \mu_2) \neq (\mu_3, \mu_4)$ in $E \times E$, the corresponding convex combinations $\lambda = (1 - b)\mu_1 + b\mu_2$ and $\beta = (1 - b)\mu_3 + b\mu_4$ in $E_b$ are $(n, \varepsilon_0)$-separated (see Figure \ref{fig:convexity-prob-measure-space}). In particular, this implies that the map $(\mu_1, \mu_2) \mapsto (1-b)\mu_1 + b\mu_2$ is a bijection from $E \times E$ onto $E_b$, so $\#E_b = (\#E)^2$, and that $E_b$ is an $(n, \varepsilon_0)$-separated subset of $\mathcal{P}(M)$. We divide the verification into two cases:
    \begin{enumerate}
        \item[i)] If $\mu_1 \neq \mu_3$, then for any $0 \leq j \leq n - 1$, by the triangle inequality, we have
        \begin{equation*}
            \rho(f_*^j\lambda, f_*^j\beta) \geq \rho(f_*^j\mu_1, f_*^j\mu_3) - \rho(f_*^j\mu_1, f_*^j\lambda) - \rho(f_*^j\beta, f_*^j\mu_3).
        \end{equation*}
        Observe that $\rho(f_*^j\mu_1, f_*^j\lambda) = \rho(f_*^j\mu_1, (1 - b)f_*^j\mu_1 + bf_*^j\mu_2)$. Thus,
        \begin{equation*}
             \rho(f_*^j\mu_1, f_*^j\lambda) = \sum\limits_{k \in \mathbb{N}} \frac{\lvert \int \varphi_k \, df^j_*\mu_1 - (1 - b)\int \varphi_k \, df^j_*\mu_1 - b\int \varphi_k \, df^j_*\mu_2\rvert}{2^k} = b \cdot \rho(f^j_*\mu_1, f^j_*\mu_2).
        \end{equation*}
        Hence, since $\operatorname{diam}(\mathcal{P}(M)) \leq 2$, we have $\rho(f_*^j\mu_1, f_*^j\lambda) \leq 2b$ for any $0 \leq j \leq n - 1$. The same computation shows that $\rho(f_*^j\beta, f_*^j\mu_3) \leq 2b$. Choosing $0 \leq \ell \leq n - 1$ such that $\rho(f_*^\ell\mu_1, f_*^\ell\mu_3) \geq \varepsilon$, we obtain $\rho(f_*^\ell\lambda, f_*^\ell\beta) \geq \varepsilon - 4b \geq b\varepsilon > \varepsilon_0$.

        \item[ii)] If $\mu_1 = \mu_3$, then $\mu_2 \neq \mu_4$. A direct computation shows that, for all $0 \leq j \leq n - 1$,
        \begin{equation*}
             \rho(f^j_*\lambda, f^j_*\beta) = b\cdot \rho(f^j_*\mu_2, f^j_*\mu_4).
        \end{equation*}
        Therefore, choosing $0 \leq \ell \leq n - 1$ such that $\rho(f_*^\ell\mu_2, f_*^\ell\mu_4) \geq \varepsilon$, we obtain $\rho(f^\ell_*\lambda, f^\ell_*\beta) \geq b\varepsilon > \varepsilon_0$.
    \end{enumerate}

    In either case, there exists $0 \leq \ell \leq n - 1$ such that $\rho(f^\ell_*\lambda, f^\ell_*\beta) \geq \varepsilon_0$. Thus, $E_b$ is $(n, \varepsilon_0)$-separated and $\operatorname{Sep}(f_*, n, \varepsilon_0) \geq \#E_b = (\#E)^2 = \operatorname{Sep}(f_*, n, \varepsilon)^2$.
\end{proof}

\begin{figure}[ht]
    \centering
   \begin{tikzpicture}[scale=1]
        \draw[line width=1pt] (0,0) -- (10,0);
        \draw[line width=1pt] (0,0) -- (5,5);
        \draw[line width=1pt] (5,5) -- (10,0);

        \draw[line width=1pt, red] (1,1) -- (6,4);
        \draw[line width=1pt, red] (6,4) -- (9,0);
        \draw[line width=1pt, red] (1,1) -- (9,0); 

        \fill[blue] (2,1.6) circle (3pt);
        \fill[blue] (2,0.85) circle (3pt);

        \fill[blue] (5,3.4) circle (3pt);
        \fill[blue] (6.7,3) circle (3pt);

        \fill[blue] (8.25,1) circle (3pt);
        \fill[blue] (8,0.15) circle (3pt);      
        
        \fill (0,0) circle (3pt);
        \node[below=7pt] at (0,0) {$\mu_1$};

        \fill (10,0) circle (3pt);
        \node[below=7pt] at (10,0) {$\mu_3$};

        \fill (5,5) circle (3pt);
        \node[right=7pt] at (5.1,5) {$\mu_2$};

        \fill[red] (4,4) circle (3pt);
        \node[left=7pt] at (4,4.1) {$\beta_3$};

        \fill[red] (1,1) circle (3pt);
        \node[left=7pt] at (1,1.1) {$\beta_2$};

        \fill[red] (1,0) circle (3pt);
        \node[below=7pt] at (1,0) {$\beta_1$};

        \fill[red] (9,0) circle (3pt);
        \node[below=7pt] at (9,0) {$\beta_6$};

        \fill[red] (9,1) circle (3pt);
        \node[right=7pt] at (9.1,1) {$\beta_5$};

        \fill[red] (6,4) circle (3pt);
        \node[right=7pt] at (6.1,4) {$\beta_4$};
    \end{tikzpicture}
    \caption{Example of the construction of the set $E_b$. Here, the measures $\mu_i$ belong to $E$, and the measures $\beta_i$ belong to $E_b$. Observe that the construction of $E_b$ is based on the fact that the space of probability measures is a convex set.}
    \label{fig:convexity-prob-measure-space}
\end{figure} 

The proof of the above lemma follows the argument of a proposition in \cite{Sigmund_1978}, where Sigmund proves that $h(f_*) \in \{0, \infty\}$. Moreover, observe that the $\delta$ defined above depends only on $\varepsilon$. In what follows, let $N_{f_*} = \{[\operatorname{Sep}(f_*, n, \varepsilon)] \in \mathbb{O} : \varepsilon > 0\} \subset \mathbb{O}$.

\begin{claim}\label{claim:roots-are-also-upper-bound}
    If $[b(n)] \in \mathbb{O}$ is an upper bound for $N_{f_*}$, then, for every $k \in \N$, $[b(n)^{2^{-k}}] \in \mathbb{O}$ is also an upper bound for $N_{f_*}$.
\end{claim}
\begin{proof}
    We know from Lemma \ref{lemma:sigmund-78-square-separate} that $\operatorname{Sep}(f_*, n, \varepsilon_0) \geq \operatorname{Sep}(f_*, n, \varepsilon)^2$ for every $n \in \N$, every $\varepsilon > 0$, and all sufficiently small $0 < \varepsilon_0 < \delta(\varepsilon)$. By our hypothesis, there exists $C_\varepsilon > 0$ such that $$C_\varepsilon b(n) \geq \operatorname{Sep}(f_*, n, \varepsilon_0) \geq \operatorname{Sep}(f_*, n, \varepsilon)^2.$$ Thus, $C_\varepsilon^{\frac{1}{2}}b(n)^{\frac{1}{2}} \geq \operatorname{Sep}(f_*, n, \varepsilon)$ for every $\varepsilon > 0$ and $n \in \N$. In other words, $[b(n)^{\frac{1}{2}}]$ is an upper bound for $N_{f_*}$. Note that $b(n)^{\frac{1}{2}}$ is still a non-decreasing sequence. Iterating this argument $k$ times, we obtain the desired property.
\end{proof}

\begin{proof}[Proof of Theorem \ref{thm:supremum-power-set-measure}, part 1]
    Assume that $o(f_*) \in \mathbb{O}$. That is, $[b(n)] = \sup N_{f_*}$ for some $[b(n)] \in \mathbb{O}$. By Claim \ref{claim:roots-are-also-upper-bound}, $[b(n)^{\frac{1}{2}}]$ is also an upper bound for $N_{f_*}$. However, since $[b(n)]$ is the least upper bound for $N_{f_*}$, it lies below every upper bound for $N_{f_*}$; in particular, $[b(n)] \leq [b(n)^{\frac{1}{2}}]$. Hence, there exists $C > 0$ such that $b(n) \leq C b(n)^{\frac{1}{2}}$. This implies $b(n) \leq C^2$ for every $n \in \N$. Therefore, $o(f_*) = [b(n)] = \ordzero$. Since $o(f_*) \geq o(f)$, it follows that $o(f) = \ordzero$. In other words, if $o(f) > \ordzero$, then $o(f_*) \in \ordgrow\setminus\mathbb{O}$. 
    
    Conversely, the condition $o(f) = \ordzero$ is equivalent to $f$ being Lyapunov-stable \cite[Theorem 3]{Correa_Pujals_2023}. Moreover, Sigmund proved in \cite[p. 420]{Sigmund_1978} that if $f$ is Lyapunov-stable, then so is $f_*$. Therefore, applying \cite[Theorem 3]{Correa_Pujals_2023} again, we obtain that if $o(f) = \ordzero$, then $o(f_*) = \ordzero$.
\end{proof}

\begin{proof}[Proof of Theorem \ref{thm:supremum-power-set-measure}, part 3]
    Suppose that $[b(n)] \geq o(f_*)$. This implies that $[b(n)]$ is an upper bound for $N_{f_*}$. Given $t > 0$, let $k \in \N$ be such that $t \geq 2^{-k}$. By Claim \ref{claim:roots-are-also-upper-bound}, $[b(n)^{2^{-k}}]$ is also an upper bound for $N_{f_*}$. Consequently, $[b(n)^t]$ is an upper bound for $N_{f_*}$ for every $t > 0$. Therefore, $\inf \mathbb{B} \geq o(f_*)$.
\end{proof}

\begin{proof}[Proof of Theorem \ref{thm:supremum-power-set-measure}, part 2]
    Assume that $o(f_*) \geq [a(n)]$ for some $[a(n)] \in \mathbb{O}$. If $[b(n)]$ is an upper bound for $N_{f_*}$, then $[b(n)] \geq o(f_*)$. In particular, $[b(n)] \geq [a(n)]$. By Claim \ref{claim:roots-are-also-upper-bound}, for every $k \in \N$, $[b(n)^{2^{-k}}]$ is also an upper bound for $N_{f_*}$, and hence $[b(n)^{2^{-k}}] \geq [a(n)]$. More precisely, there exists $C > 0$ such that $Cb(n)^{2^{-k}} \geq a(n)$. Thus, $C^{2^k}b(n) \geq a(n)^{2^k}$. That is, $[b(n)] \geq [a(n)^{2^k}]$. Since $[b(n)] \in \mathbb{O}$ is an arbitrary upper bound for $N_{f_*}$ and $o(f_*) \in \ordgrow$ is the infimum of all upper bounds for $N_{f_*}$, it follows that $o(f_*) \geq [a(n)^{2^k}]$. Therefore, $o(f_*) \geq \sup \mathbb{A}$.
\end{proof}

\begin{proof}[Proof of Theorem \ref{thm:supremum-power-set-measure}, part 4]
    By the first statement of Proposition \ref{prop:measure-induced-properties}, $o(f_*) \geq o(f)$. If $o(f) \geq [a(n)]$, the second part of this theorem implies that $o(f_*) \geq \sup \mathbb{A}$. Moreover, if $o(f) \in \mathbb{O}$ and $o(f) > \ordzero$, then $o(f_*) \in \ordgrow\setminus\mathbb{O}$. Therefore, $o(f) \neq o(f_*)$, which implies $o(f_*) > o(f)$.
\end{proof}

Furthermore, the proof of Corollary \ref{coro:alpha-limit-sup-poly} follows immediately from Proposition \ref{prop:homeomorphism-linear-growth} and Theorem \ref{thm:supremum-power-set-measure}. We now prove Corollary \ref{coro:zero-or-infinite-entropy}.

\begin{proof}[Proof of Corollary \ref{coro:zero-or-infinite-entropy}]
    If $I_\mathbb{A}(f_*) = \varnothing$, then $h_\mathbb{A}(f_*) = \infty$. Otherwise, there exists a $t$ such that $o(f_*) \le [a(n)^t]$, and by the third part of Theorem \ref{thm:supremum-power-set-measure}, $o(f_*) \le \inf \mathbb{A}$. Hence, $I_\mathbb{A}(f_*) = (0,\infty)$ and $h_\mathbb{A}(f_*) = 0$.
\end{proof}

%%%%%%%%%%%%%%%%%%%%%%%%%%%%%%%%%%%%%%%%%%%%%%%%%%%%%%%%%%%%%%%%%%%%%%%%%%%%%%%%%%%%%%%%%%%%%%%%%%%%%%%%%%%%

\subsection{Proof of Theorem \ref{thm:omega-finite-measure-induced-sup-poly}}\label{subsec:proof-thm-B}

From now on, assume that $f: M \to M$ is a homeomorphism such that its set of nonwandering points is a finite set of fixed points. That is, $\Omega(f) = \operatorname{Fix}(f) = \{p_1, \ldots, p_m\}$. We make this assumption without loss of generality because a finite nonwandering set consists entirely of periodic points, meaning that under some iterate $f^n$, all periodic points become fixed. Moreover, we know by Remark \ref{rmk:iterations-measure-induced-map} that $o((f^n)_*) = o((f_*)^n)$, and, by Remark \ref{rmk:iterate-f-generalized-entropy}, that $o(f_*) \leq o((f_*)^n)$. Therefore, showing that $o((f_*)^n) \leq \sup \mathbb{P}$ is sufficient.

Let $s = \min_{i \neq j} d(p_i, p_j)$. If $m = 1$, set $s = \operatorname{diam}(M) + 1$. Given $\varepsilon > 0$, choose $K \in \N$ such that $2^{1 - K} < \varepsilon/4$. Consider the first $K$ functions $\{\varphi_1, \dots, \varphi_K\}$ of the dense sequence $(\varphi_k)_{k \in \mathbb{N}}$ used to define the metric $\rho$ in \S\ref{subsec:probability-space-measure-induced}. Moreover, choose $\eta \in (0, s/2)$ such that
\begin{enumerate}
    \item[(S1)] if $d(x, y) < \eta$, then $d(f(x), f(y)) < s/2$, which is possible by the continuity of $f$;
    \item[(S2)] for any $A \subset M$, if $\diam(A) < \eta$, then $\operatorname{osc}_A \varphi_i < \varepsilon/8$ for all $1 \leq i \leq K$.
\end{enumerate}
In particular, $\eta + s/2 < s$. Set $V_j = \{x \in M : d(p_j, x) < \eta/3\}$ and $V = \bigcup_j V_j$. Note that $V_i \cap V_j = \varnothing$ for $i \neq j$ by the choice of $\eta$, and that $V$ is an open set. Hence, $B := M\setminus V$ is a compact set contained in $M\setminus\Omega(f)$.

\begin{remark}\label{rmk:no-jump}
If $f^k(x)\in V_i$ and $f^{k+1}(x)\in V_j$, then $i=j$. Indeed, $f^k(x)\in V_i$ gives $d(f^k(x),p_i)<\eta/3<\eta$, so by (S1) and $f(p_i)=p_i$ we get $d(f^{k+1}(x),p_i)<s/2$. Since also $d(f^{k+1}(x),p_j)<\eta/3$, we obtain $d(p_i,p_j)<s/2+\eta/3<s$, which forces $i=j$.
\end{remark}

\begin{claim}\label{claim:cover-wandering-closed-set}
    There exists $L \in \N$, depending on $\varepsilon$, such that for every $x\in M$, $\#\{k\ge0 : f^k(x)\in B\} \le L$. More precisely, we can cover $B$ with open sets $U_i$, $1 \le i \le L$, which are visited at most once by the orbit of any $x \in M$.
\end{claim}
\begin{proof}
    Given $y \in B$, there exists an open neighborhood $U_y$ with $f^n(U_y) \cap U_y = \varnothing$ for all $n \geq 1$. Since subsets of wandering sets are also wandering, we can shrink $U_y$ so that $\diam(U_y) < \eta$. Since $B$ is compact, there exist finitely many such wandering neighborhoods $U_1, \ldots, U_L$ that cover $B$. For any $x \in M$, the orbit of $x$ meets each $U_i$ at most once: if $f^{n_1}(x),f^{n_2}(x)\in U_i$ with $n_1<n_2$, then $f^{n_2-n_1}(U_i)\cap U_i\neq\varnothing$, which is a contradiction. Hence $f^k(x)\in B\subset\bigcup_i U_i$ for at most $L$ values of $k$.
\end{proof}

Set $E_i = U_i \setminus (V \cup U_1 \cup \dots \cup U_{i - 1})$ for $1 \le i \leq L$, where the open sets $U_i$ are given in Claim \ref{claim:cover-wandering-closed-set}. Thus, $\mathcal{Q} = \{E_1, \dots, E_L, V_1, \dots, V_m\}$ is a partition of $M$ into $L + m$ Borel measurable sets, each with a diameter less than $\eta$. Observe that $B = \cup_{i = 1}^L E_i$.

\begin{claim}\label{claim:measure-B-decreases-along-time}
    If $\mu \in \mathcal{P}(M)$, then for every $n \in \N$, $\sum_{k < n} f_*^k\mu(B) \le L$.
\end{claim}
\begin{proof}
    By Claim \ref{claim:cover-wandering-closed-set}, for every $x\in M$, we have $\sum_{k<n}\mathbf{1}_B(f^k(x))\le L$. Since $\mathbf{1}_B$ is a measurable function and $f_*^k\mu(B)=\int \mathbf{1}_B\circ f^k\,d\mu$, basic properties of integration give
    \begin{equation*}
        \sum_{k<n}f_*^k\mu(B)=\int\sum_{k<n}\mathbf{1}_B(f^k(x))\,d\mu(x)\le L.\qedhere
    \end{equation*}
\end{proof}

For the following, set $\delta_0 = \frac{\varepsilon}{4(L + m)}$. And recall that $\rho$ is a metric on $\mathcal{P}(M)$ defined on \S\ref{subsec:probability-space-measure-induced}.

\begin{claim}\label{claim:same-profile-not-separated}
    If $\mu, \nu \in \mathcal{P}(M)$ satisfy $\abs{\mu(Q) - \nu(Q)} < \delta_0$ for every $Q \in \mathcal{Q}$, then $\rho(\mu, \nu) < \varepsilon$.
\end{claim}
\begin{proof}
    Fix some $i \leq K$ and choose $y_Q \in Q$ for every $Q \in \mathcal{Q}$. Since $\diam(Q) < \eta$, we have $\operatorname{osc}_Q \varphi_i < \varepsilon/8$. Hence, for any $x \in Q$, $\abs{\varphi_i(x) - \varphi_i(y_Q)} < \varepsilon/8$. Thus,
    \begin{equation*}
        \abs{\int \varphi_i d\mu - \sum\limits_{Q \in \mathcal{Q}} \varphi_i(y_Q)\mu(Q)} \leq \sum\limits_{Q \in \mathcal{Q}} \int_Q \abs{\varphi_i - \varphi_i(y_Q)} d\mu \leq \frac{\varepsilon}{8}.
    \end{equation*}
    The same estimate holds for $\nu$. Since $\abs{\varphi_i(y_Q)} \leq 1$, the triangle inequality yields
    \begin{equation*}
        \abs{\int\varphi_i\,d\mu-\int\varphi_i\,d\nu}\le\frac{\varepsilon}{8}+\frac{\varepsilon}{8}+\sum_{Q \in \mathcal{Q}}\abs{\mu(Q)-\nu(Q)}\le\frac{\varepsilon}{4}+(L + m)\delta_0=\frac{\varepsilon}{2}.
    \end{equation*}
    For $i > K$, the tail of the series in the definition of the metric $\rho$ is bounded by $2^{1 - K}$. Therefore,
    \begin{equation*}
        \rho(\mu, \nu) = \sum\limits_{i \in \N} \frac{\abs{\int \varphi_i d\mu - \int \varphi_i d\nu}}{2^i}\le\sum_{i\le K}\frac{1}{2^i}\cdot\frac{\varepsilon}{2}+\sum_{i>K}\frac{2}{2^i}<\frac{\varepsilon}{2}+2^{1-K}<\frac{\varepsilon}{2}+\frac{\varepsilon}{4}<\varepsilon.\ \qedhere
    \end{equation*}
\end{proof}

\begin{claim}\label{claim:update-time}
    For every $\mu \in \mathcal{P}(M)$, $1 \le j \le m$, and $n \in \N$, we have
    \begin{equation*}
        \sum\limits_{k < n - 1} \abs{f_*^k\mu(V_j) - f_*^{k + 1}\mu(V_j)} \le 2L.
    \end{equation*}
\end{claim}
\begin{proof}
    Fix $\mu$ and $j$, and let $\alpha_k=f_*^k\mu(V_j)$ and $\beta_k=f_*^k\mu(B)$. Let
    \begin{align*}
        P_k&=\mu(\{x \in M: f^k(x)\notin V_j,\ f^{k+1}(x)\in V_j\}),\\
        R_k&=\mu(\{x \in M: f^k(x)\in V_j,\ f^{k+1}(x)\notin V_j\}).
    \end{align*}
    Then $\alpha_{k+1}-\alpha_k=P_k-R_k$. Indeed, we have
    \begin{equation*}
        \alpha_{k + 1} - \alpha_k = f^{k+1}_*\mu(V_j) - f^k_*\mu(V_j) = f^k_*\mu(f^{-1}V_j\setminus V_j) - f^k_*\mu(V_j\setminus f^{-1}V_j),
    \end{equation*}
    which yields exactly $P_k - R_k$. By Remark~\ref{rmk:no-jump}, if $f^k(x)\notin V_j$ and $f^{k+1}(x)\in V_j$, then $f^k(x)\in B$, because $M = V \sqcup B$; hence $P_k\le f_*^k\mu(B)=\beta_k$. Likewise, if $f^k(x)\in V_j$ and $f^{k+1}(x)\notin V_j$, then $f^{k+1}(x)\in B$; hence $R_k\le\beta_{k+1}$. Since $P_k,R_k\ge0$, we get $\abs{\alpha_{k+1}-\alpha_k}\le\beta_k+\beta_{k+1}$. Summing over $k<n-1$ and applying Claim \ref{claim:measure-B-decreases-along-time}, we obtain
    \begin{equation*}
        \sum\limits_{k < n - 1} \abs{f_*^k\mu(V_j) - f_*^{k + 1}\mu(V_j)}\le\sum_{k<n-1}(\beta_k+\beta_{k+1})\le 2L.\qedhere
    \end{equation*}
\end{proof}

Set $\varepsilon_0 = \delta_0/2$ and $\varepsilon_0\Z = \{\varepsilon_0k : k \in \Z\}$. Now we define $\mathfrak{p}_\mu(n)$, the \emph{profile of $\mu$} with respect to the partition $\mathcal{Q}$. For each $n \in \mathbb{N}$, we associate to $\mu$ an $n\times (L+m)$ matrix in the following manner:
\begin{itemize}
    \item[(P1)] For $E_i \in \mathcal{Q}$, where $1 \le i \le L$, set $a_{k, i} = \floor{\frac{f^k_*\mu(E_i)}{\varepsilon_0}}$ for $0 \le k \le n - 1$.
    \item[(P2)] For $V_j \in \mathcal{Q}$, where $1 \le j \le m$, we define the entries $a_{k, L+j}$ recursively for $0 \le k \le n -1$. First, set $a_{0, L+j}$ as the point in $\varepsilon_0\Z$ closest to $\mu(V_j)$ (choosing the smaller value in the event of a tie). That is, $\abs{\mu(V_j) - a_{0, L+j}} \le \varepsilon_0/2$. Given $a_{k, L+j}$, if $\abs{f^{k + 1}_*\mu(V_j) - a_{k, L+j}} < \varepsilon_0$, set $a_{k + 1, L+j} = a_{k, L+j}$; otherwise, set $a_{k + 1, L+j}$ as the point in $\varepsilon_0\Z$ closest to $f^{k + 1}_*\mu(V_j)$ (again choosing the smaller value if tied). By construction, $\abs{f^k_*\mu(V_j) - a_{k, L+j}} < \varepsilon_0$.
\end{itemize}

In case (P2), if $a_{k + 1, L + j}$ differs from $a_{k, L + j}$, we call $k + 1$ an \emph{update time}. Before proving the next lemma, observe that the set of possible profiles is determined by $\varepsilon$ and $n$.

\begin{lemma}
    Given $\varepsilon > 0$, the number of distinct possible profiles of measures in $\mathcal{P}(M)$ grows at most polynomially in $n$. That is, there exist constants $C_\varepsilon, t(\varepsilon) > 0$ such that $\#\{\mathfrak{p}_\mu(n) : \mu \in \mathcal{P}(M)\} \leq C_\varepsilon n^{t(\varepsilon)}$.
\end{lemma}
\begin{proof}
    By Claim \ref{claim:measure-B-decreases-along-time} (using the fact that $B = \bigcup_{i = 1}^L E_i$), we have
    \begin{equation*}
        \sum\limits_{k = 0}^{n - 1}\sum\limits_{i = 1}^L a_{k, i} \leq \frac{1}{\varepsilon_0}\sum\limits_{k = 0}^{n - 1} f^k_*\mu(B) \leq \frac{L}{\varepsilon_0}.
    \end{equation*}
    Let $L_0 = \floor{\frac{L}{\varepsilon_0}}$. Since each $a_{k, i}$ is a nonnegative integer, the sum is also less than or equal to $L_0$. Hence, the number of possibilities to sum $nL$ nonnegative integers to obtain at most $L_0$ is $\binom{nL + L_0}{L_0} \leq (nL + L_0)^{L_0}$.

    Moreover, let $k_1 < k_2$ be two consecutive update times for some index $1 \le j \le m$. By definition, $\abs{f^{k_1}_*\mu(V_j) - a_{k_1, L+j}} \le \varepsilon_0/2$. Since $k_2$ is an update time, $\abs{f^{k_2}_*\mu(V_j) - a_{k_1, L+j}} \ge \varepsilon_0$, which implies $\abs{f^{k_1}_*\mu(V_j) - f^{k_2}_*\mu(V_j)} \geq \varepsilon_0/2$. By Claim \ref{claim:update-time}, the number of update times is at most $\ell := 1 + \floor{4L/\varepsilon_0}$, so there are at most $\ell + 1$ possibilities for the number of update times. Since each value $a_{k, L + j}$ belongs to $\varepsilon_0\Z \cap [-\varepsilon_0, 1+\varepsilon_0]$, there are at most $1/\varepsilon_0 + 3$ possibilities for this value. Thus, for each $j$, there are at most $(\ell + 1)(n + 1)^\ell(1/\varepsilon_0 + 3)^{\ell + 1}$ possibilities for the finite sequence $(a_{0, L + j}, \ldots, a_{n - 1, L + j})$ determined by the update times and their corresponding values. Therefore, using that $n + 1 \le 2n$,
    \begin{equation*}
        \#\{\mathfrak{p}_\mu(n) : \mu \in \mathcal{P}(M)\} \leq (nL + L_0)^{L_0} \cdot \left( (\ell + 1)(n + 1)^\ell\paren{\frac{1}{\varepsilon_0} + 3}^{(\ell + 1)} \right)^m \leq C_\varepsilon n^{t(\varepsilon)},
    \end{equation*}
    where $t(\varepsilon) = L_0 + m\ell$ and $C_\varepsilon = (L+L_0)^{L_0} (\ell + 1)^m 2^{m\ell}\paren{\frac{1}{\varepsilon_0}+3}^{m(\ell + 1)}$.
\end{proof}

\begin{proof}[Proof of Theorem \ref{thm:omega-finite-measure-induced-sup-poly}]
    Given $\varepsilon > 0$ and $n \in \N$, let $A \subset \mathcal{P}(M)$ be an $(n, \varepsilon)$-separated set such that $\#A = \operatorname{Sep}(f_*, n, \varepsilon)$. Suppose that $\mathfrak{p}_\mu(n) = \mathfrak{p}_\nu(n)$ for some $\mu, \nu \in A$. This implies that, for $0 \le k \le n - 1$, $1 \le i \le L$, and $1 \le j \le m$, we have $a_{k, i} = a'_{k,i}$ and $a_{k, L + j} = a'_{k, L + j}$, where $a$ and $a'$ are the entries of $\mathfrak{p}_\mu(n)$ and $\mathfrak{p}_\nu(n)$, respectively. For the first family, which corresponds to (P1), we have
    \begin{equation*}
        \floor{\frac{f^k_*\mu(E_i)}{\varepsilon_0}} = \floor{\frac{f^k_*\nu(E_i)}{\varepsilon_0}} \implies \abs{f^k_*\mu(E_i) - f^k_*\nu(E_i)} < \varepsilon_0 < \delta_0.
    \end{equation*}
    Furthermore, $a_{k, L + j} = a'_{k, L + j}$ implies $\abs{f^k_*\mu(V_j) - f^k_*\nu(V_j)} < 2\varepsilon_0 = \delta_0$. Thus, by Claim \ref{claim:same-profile-not-separated}, for every $0 \leq k \leq n - 1$, $\rho(f^k_*\mu, f^k_*\nu) < \varepsilon$. However, since $A$ is an $(n, \varepsilon)$-separated set, this forces $\mu = \nu$. Therefore, the map $\mu \mapsto \mathfrak{p}_\mu(n)$ is an injective function from $A$ to $\{\mathfrak{p}_\mu(n) : \mu \in \mathcal{P}(M)\}$, which implies $\operatorname{Sep}(f_*, n, \varepsilon) \leq C_\varepsilon n^{t(\varepsilon)}$. Therefore,
    \begin{equation*}
        o(f_*) = \sup \{[\operatorname{Sep}(f_*, n, \varepsilon)] \in \mathbb{O} : \varepsilon > 0\} \leq \sup \{[n^t] \in \mathbb{O} : t > 0\} = \sup \mathbb{P}. \qedhere
    \end{equation*}
\end{proof}

\printbibliography

@article{Correa_Pujals_2023,
    author = {J. Correa and E. R. Pujals},
    title = {Orders of growth and generalized entropy},
    journal = {J. Inst. Math. Jussieu},
    volume = {22},
    number = {4},
    pages = {1581--1613},
    year = {2023},
    doi = {10.1017/S1474748021000463}
}

@incollection{Sigmund_1978,
    author = {K. Sigmund},
    title = {Affine transformations on the space of probability measures},
    booktitle = {Systèmes dynamiques III - Varsovie},
    series = {Astérisque},
    publisher = {Société mathématique de France},
    number = {51},
    pages = {415--427},
    year = {1978},
    mrnumber = {480946},
    zbl = {0385.28010}
}

@article{Glasner_Weiss_1995,
    author = {E. Glasner and B. Weiss},
    title = {Quasi-Factors of Zero Entropy Systems},
    journal = {J. Am. Math. Soc.},
    volume = {8},
    number = {3},
    pages = {665--686},
    year = {1995},
    publisher = {American Mathematical Society},
    issn = {08940347, 10886834},
    doi = {10.2307/2152926}
}

@article{Bauer_Sigmund_1975,
    author = {W. Bauer and K. Sigmund},
    title = {Topological dynamics of transformations induced on the space of probability measures},
    journal = {Monatsh. Math.},
    volume = {79},
    pages = {81--92},
    year = {1975},
    doi = {10.1007/BF01585664}
}

@article{Palis_1969,
    author = {J. Palis},
    title = {On Morse-Smale dynamical systems},
    journal = {Topology},
    volume = {8},
    number = {4},
    pages = {385--404},
    year = {1969},
    issn = {0040-9383},
    doi = {10.1016/0040-9383(69)90024-X}
}

@article{Correa_dePaula_2023,
    author = {J. Correa and H. {de Paula}},
    title = {Polynomial entropy of Morse-Smale diffeomorphisms on surfaces},
    journal = {Bull. Sci. Math.},
    volume = {182},
    pages = {103225},
    year = {2023},
    issn = {0007-4497},
    doi = {10.1016/j.bulsci.2022.103225}
}

@article{Marco_2013,
    author = {J. P. Marco},
    title = {Polynomial entropies and integrable Hamiltonian systems},
    journal = {Regul. Chaot. Dyn.},
    volume = {18},
    pages = {623--655},
    year = {2013},
    doi = {10.1134/S1560354713060051}
}

@book{Denker_Grillenberger_Sigmund_1976,
    author = {M. Denker and C. Grillenberger and K. Sigmund},
    title = {Ergodic Theory on Compact Spaces},
    series = {Lecture Notes in Mathematics},
    volume = {527},
    publisher = {Springer-Verlag},
    year = {1976},
    isbn = {3-540-07797-9}
}

@book{Katok_Hasselblatt_1995,
    author = {A. Katok and B. Hasselblatt},
    title = {Introduction to the Modern Theory of Dynamical Systems},
    series = {Encyclopedia of Mathematics and its Applications},
    volume = {54},
    publisher = {Cambridge University Press},
    year = {1995},
    isbn = {9780521341875},
    doi = {10.1017/CBO9780511809187}
}

@article{Labrousse_Marco_2014,
    author = {C. Labrousse and J. P. Marco},
    title = {Polynomial entropies for Bott integrable Hamiltonian systems},
    journal = {Regul. Chaot. Dyn.},
    volume = {19},
    pages = {374--414},
    year = {2014},
    doi = {10.1134/S1560354714030083}
}

\end{document}